\documentclass[reqno, 11pt]{amsart}
\usepackage{url}
\usepackage[hidelinks]{hyperref}
\usepackage{cleveref}
 \usepackage{relsize}
 \usepackage[dvipsnames]{xcolor}
\usepackage{tikz}
\usetikzlibrary{arrows.meta, positioning, decorations.markings}
\usepackage{tikz-cd}
\usepackage{pgfplots}
\usepgfplotslibrary{groupplots}
\pgfplotsset{compat=1.18}

\usetikzlibrary{3d, decorations.pathmorphing, fadings, backgrounds,decorations.pathreplacing}
\usepackage{pgfmath}
\usepackage{tikz-3dplot}
\usepackage{enumitem}
\usepackage{helvet}
\usepackage{ stmaryrd }
\usepackage{color,soul}
\usepackage{amsthm, thmtools}
\usepackage[style=numeric,maxbibnames=99]{biblatex}
\usepackage{caption} 
\usepackage{float}
\newtheorem{mainthm}{Theorem}

\newtheorem{maincor}[mainthm]{Corollary}

\newtheorem{theorem}{Theorem}[section]
\newtheorem{lemma}[theorem]{Lemma}
\newtheorem{prop}[theorem]{Proposition}

\newtheorem*{claim}{Claim}
\newtheorem{cor}[theorem]{Corollary}

\newtheorem{remark}[theorem]{Remark}

\theoremstyle{definition}
\newtheorem{deff}[theorem]{Definition}

\newtheoremstyle{break}
  {\topsep}{\topsep}%
  {\itshape}{}%
  {\bfseries}{}%
  {\newline}{}%
\theoremstyle{break}

\numberwithin{equation}{section}

\usepackage{bbm}

\usepackage{euscript}

\usepackage{pb-diagram}

\usepackage{amsmath}

\usepackage{amsxtra}
\usepackage{amssymb}
\usepackage{pifont}
\usepackage{amsbsy}

\usepackage{graphicx}
\usepackage{epstopdf}
\usepackage{colortbl}
\usepackage{multirow}
\usepackage{hhline}

\newskip\aline \newskip\halfaline
\newcommand\bC{{\mathbb C}}

\newcommand{\bP}{{{\mathbb P}}}

\newcommand\bR{{\mathbb R}}

\newcommand\bS{{\mathbb S}}

\newcommand\bZ{{\mathbb Z}}

\DeclareMathOperator{\vol}{vol}

\newcommand{\eps}{{\epsilon}}

\def\inpr{\mathbin{\hbox to 6pt{\vrule height0.4pt width5pt depth0pt \kern-.4pt \vrule height6pt width0.4pt depth0pt\hss}}}

\newcommand{\Scal}{R}
\newcommand{\Ric}{\mathrm{Ric}}
\newcommand{\cA}{\mathcal{A}}
\newcommand{\Lip}{\mathrm{Lip}}
\newcommand{\cH}{\mathcal{H}}

\def\XXint#1#2#3{{\setbox0=\hbox{$#1{#2#3}{\int}$ }
\vcenter{\hbox{$#2#3$ }}\kern-.59\wd0}}

\title{Positive curvature conditions on contractible manifolds}

\author[Paul Sweeney Jr.]{Paul Sweeney Jr.}
\address{Paul Sweeney Jr., Universit\`a di Trento, Dipartimento di Matematica, via Sommarive 14, 38123 Povo di Trento, Italy}
\email{paul.sweeneyjr@unitn.it}

\begin{document}
\begin{abstract}
     Our goal is to identify curvature conditions that distinguish Euclidean space in the case of open, contractible manifolds and the disk in the case of compact, contractible manifolds with boundary. First, we show that an open manifold that is the interior of a sufficiently connected, compact, contractible 5-manifold with boundary and supports a complete Riemannian metric with uniformly positive scalar curvature is diffeomorphic to Euclidean 5-space. Next, we investigate the analogous question for compact manifolds with boundary: Must a compact, contractible manifold that supports a Riemannian metric with positive scalar curvature and mean convex boundary necessarily be the disk? We present examples demonstrating that this curvature condition alone cannot distinguish the disk; on the other hand, we exhibit stronger curvature conditions that allow us to draw such a conclusion.
    
\end{abstract}
\maketitle

\section{Introduction}
A classical theme in Riemannian geometry is that positive curvature imposes constraints on the topology of a manifold\footnote{All manifolds are assumed to be smooth and connected.}. Perhaps the best-known example of this interaction between geometry and topology is the Gauss--Bonnet theorem, which implies that the only compact 2-manifold with boundary that supports a Riemannian metric with positive (scalar) curvature such that the boundary has positive geodesic curvature is the 2-disk. Analogously, by the work of Cohn-Vossen \cite{CohnVossen} and Huber \cite{HuberOnsubharmonic}, we know that the only open\footnote{An open manifold is a non-compact manifold without boundary.} 2-manifold supporting a {complete} Riemannian metric with positive (scalar) curvature is $\bR^2.$ 

In higher dimensions, scalar curvature differs from other notions of curvature, and positive scalar curvature becomes a relatively weak condition; nonetheless, it can still provide significant topological information. Therefore, we consider the following natural questions:
\begin{enumerate}[label=\Roman*.]
    \item Let $M^{n+1}$ be an open manifold which supports a \emph{complete} Riemannian metric with positive scalar curvature. Is $M$ homeomorphic (or diffeomorphic) to the standard $\bR^{n+1}$?\label{Question1}
    \item Let $X^{n+1}$ be a compact manifold with boundary which supports a Riemannian metric with positive scalar curvature and mean convex boundary. Is $X$ homeomorphic (or diffeomorphic) to the standard $(n+1)$-disk?\label{Question2}
\end{enumerate}

In dimension three, there are partial results to \hyperref[Question1]{Question I}.  Chang, Weinberger, and Yu \cite{ChangWeinbergerYuTaming} proved that an open, \emph{contractible} 3-manifold with a complete Riemannian metric of uniformly positive scalar curvature must be diffeomorphic to $\bR^3$. More recently, Chodosh--Lai--Xu \cite{ChodoshLaiXu3Manifolds} showed an open, \emph{contractible} Riemannian 3-manifold supporting a complete Riemannian metric with nonnegateve scalar curvature and bounded geometry\footnote{A Riemannian manifold has bounded geometry if there exists a constant $\Lambda>0$ such that $|\mathrm{Riem}|\leq \Lambda$ and $\mathrm{inj}\geq \Lambda^{-1}.$} is diffeomorphic to $\bR^3$ (see also Wang \cite{WangContractibleI, WangContractibleII}). We note that $\bR^n$ does indeed support a Riemannian metric with uniformly positive scalar curvature, for instance, by smoothly capping off a round half-cylinder $[0, \infty) \times \bS^{n-1}$ with a round $n$-hemisphere. 

Furthermore, there has been partial progress on Yau's fundamental question \cite[Problem 27]{YauProblem} of classifying 3-manifolds supporting complete Riemannian metrics with positive scalar curvature. Gromov \cite{GromovFourLectures} and Wang \cite{WanguniformlyPSC3}, independently, showed an open 3-manifold supports a complete Riemannian metric with \emph{uniformly} positive scalar curvature if and only if it is an infinite connected sum of manifolds of the form $\bS^3/\Gamma_j$ and $\bS^2 \times \bS^1$, where $\Gamma_j$ is a finite
subgroup of $\mathrm{SO}(4)$ acting freely on $\bS^3$. (See also Bessi\`eres--Besson--Maillot \cite{BessieresBessoMaillotRicciflow}, Bessi\`eres--Besson--Maillot--Marques \cite{BessieresBessoMaillotMarquesDeforming}, and Dong \cite{DongThreemanifolds}.) More recently, Balacheff--Gil Moreno de Mora Sard\`a--Sabourau \cite{BalacheffMorenodeMoraSardàSabourauComplete} have shown that one can weaken the uniform positive scalar curvature hypothesis in Gromov's and Wang's result to positive scalar curvature with at most quadratic decay at infinity with constant $C > 64\pi^2$ and still have the same topological classification.

In dimension four, there is also a partial answer to \hyperref[Question1]{Question I}, which is similar in spirit to the ones above. For open 4-manifolds, Chodosh, Máximo, and Mukherjee \cite{ChodoshMaximoMukherjeeComplete4} proved that if $M$ is the interior of a compact, contractible 4-manifold with boundary $X$, and $M$ supports a complete Riemannian metric of uniformly positive scalar curvature, then $M$ is homeomorphic to $\bR^4$. Moreover, if $X$ is a Mazur manifold\footnote{A Mazur manifold is a compact, contractible smooth 4-manifold with boundary admitting a handle body decomposition with one 0-handle, one 1-handle, and one 2-handle.}, then the homeomorphism can be promoted to a diffeomorphism.

We emphasize that the \textit{completeness} of the Riemannian metric is a crucial hypothesis in these results concerning open manifolds. Indeed, for every manifold $M$ that is diffeomorphic to the interior of a compact manifold with boundary, we know, by Gromov's h-principle \cite{GromovHprinciple}, that $M$ supports a (possibly incomplete) Riemannian metric with positive sectional curvature (see also Rosenberg \cite[Theorem 0.1 and Remark 0.2]{RosenbergManifolds}). This highlights that without the completeness assumption, no meaningful topological restrictions can be expected.

In this paper, we are able to affirmatively answer the following reformulation of \hyperref[Question1]{Question I}:  Can one find a (partial) characterization of open 5-manifolds supporting complete metrics of uniformly positive scalar curvature, analogous to the 4-dimensional result of Chodosh, Máximo, and Mukherjee? 

\begin{mainthm}\label{Scalar5}
    Let $M$ be the interior of a compact, contractible 5-manifold with boundary $X$, such that $\pi_3(X, \partial X) = 0$. If $M$ supports a complete Riemannian metric of uniformly positive scalar curvature, then $M$ is {diffeomorphic} to $\bR^5$.
\end{mainthm}

\begin{remark}
     We note that the boundary of every compact, contractible manifold with boundary is a homology sphere.
     
     Kervaire \cite{KervaireSmooth} proved that every homology 4-sphere bounds a compact, contractible 5-manifold, and that every finitely presented perfect group of deficiency\footnote{The deficiency of a finitely presented group is the maximum of the difference between the number of generators and the number of relations, taken over all presentations.} zero is the fundamental group of an integral homology 4-sphere. For instance, the binary icosahedral group is a perfect group with deficiency zero. Hence, there exist compact, contractible 5-manifolds whose boundary is a non-trivial integral homology 4-sphere. Therefore, in any such case, the interior $M$ of $X$ is not $\bR^5$.
\end{remark}

 Next, we would like to consider \hyperref[Question2]{Question II}. First, we recall that interior curvature bounds alone are insufficient to characterize compact manifolds with boundary. Indeed, Gromov’s h-principle \cite{GromovHprinciple} guarantees that any compact manifold with boundary admits a Riemannian metric with positive sectional curvature. Therefore, we must impose an additional hypothesis on the boundary. 
 
 In dimension 3, Carlotto--Li \cite{CarlottoLiConstrained, CarlottoLiConstrainedII} classified the manifolds that satisfy the hypotheses of \hyperref[Question2]{Question II}. Specifically, they showed a connected, orientable, compact 3-manifold supports a Riemannian metric with positive scalar curvature and mean convex boundary if and only if it is the (interior) connected sum of the form $(\#_{i=1}^I P_{\gamma_i})\#(\#_{j=1}^J \bS^3/\Gamma_j)\# (\#_{k=1}^K  \bS^2\times \bS^1)$, where $\Gamma_j$ is a finite
subgroup of $\mathrm{SO}(4)$ acting freely on $\bS^3$ and $P_{\gamma_i}$ is a genus $\gamma_i$ handlebodies. Therefore, the 3-disk is the only compact, \emph{contractible} 3-manifold that supports a Riemannian metric with positive scalar curvature and mean convex boundary. (See also the work of Wu \cite{WuCapillary} for constraints on Riemannian 2- and 3-manifolds with boundary that have nonnegative scalar curvature and mean convex boundary.)

The situation differs, substantially, in higher dimensions and from the presentation concerning \hyperref[Question1]{Question I}. In part, this is due to a result of Lawson and Michelsohn \cite[\S 5]{LawsonMichelsohnEmbedding}, which implies the following:
\begin{prop}\label{Example1}
    Let $X^{n+1}$, $n\geq 2$, be a compact, contractible $(n+1)$-manifold with boundary. If $n=3$, additionally assume that $X$ is a Mazur manifold. Then $X$ supports a Riemannian metric with positive scalar curvature and mean convex boundary.
\end{prop}
 
\begin{remark}
    By Kervaire \cite{KervaireSmooth}, for any oriented integral homology $n$-sphere $\Sigma^n$ with $n\geq 5$, there exists a unique smooth homotopy sphere $S_M^n$ such that $\Sigma^n\# S_M^n$  bounds a compact, contractible manifold. Furthermore, for every finitely presented superperfect group\footnote{A group $G$ is superperfect if its first two homology groups are trivial.} $G$ there exists an integral homology whose fundamental group is $G$. For example, the binary icosahedral group is a finitely presented superperfect group. Moreover, when the smooth generalized Poincar\'e conjecture holds (such as in dimension 5), we know that $S_M^n$ is diffeomorphic to the standard $\bS^n$ and so $\Sigma^n$ is diffeomorphic to $\Sigma^n\# S_M^n$. Therefore, by combining the work of Kervaire with \Cref{Example1}, we can find a plethora of examples of compact, contractible manifolds that support positive scalar curvature and mean convex boundary.
\end{remark}

In fact, the manifolds from \Cref{Example1} support metrics with even stronger curvature conditions. By Lawson and Michelsohn’s result, each such manifold supports a Riemannian metric with  \emph{constant sectional curvature (equal to one)} and mean convex boundary. Furthermore, applying the work of B\"ar--Hanke \cite[Theorem 3.7]{BarHankeBoundary}  to \Cref{Example1} shows that the manifolds from \Cref{Example1} support Riemannian metrics with positive scalar curvature and \emph{convex boundary}. Therefore, to obtain a characterization of the standard disk among compact, contractible manifolds with boundary, it is necessary to impose significantly stronger geometric assumptions.

With \emph{very strong} curvature conditions, there have been affirmative results concerning \hyperref[Question2]{Question~II}; the first of which is known as the Soul Theorem. Gromoll and Meyer \cite{GromollMeyer} (see also Cheeger--Gromoll \cite{CheegerGromollSoul}, Poor \cite{PoorSomeresults}, and Perelman \cite{PerelmanProofofSoul}) showed that the only compact $n$-manifold with boundary that supports a Riemannian metric with positive sectional curvature and convex boundary is the $n$-disk. Sha proved an analogous result for positive sectional curvature and $p$-convex boundary \cite{Shapconvex, ShaHandlebodiespconvex}. An even stronger result exists in dimension three: the only compact manifold with boundary that supports a Riemannian metric with positive Ricci curvature and convex boundary is the $3$-disk. This follows from a variational argument developed in Meeks, Simon, and Yau \cite{MeeksSimonYauEmbedding} (see Fraser--Li \cite[Theorem 2.11]{FraserLicompact3}).

Therefore, in the present work, we aim to investigate \hyperref[Question2]{Question~II} under curvature conditions that are stronger than the combination of positive scalar curvature on the interior and mean convexity on the boundary, yet weaker than positive sectional curvature on the interior and convexity on the boundary, with the goal of identifying conditions that distinguish the disk.

 An interior curvature condition stronger than positive scalar curvature is positive isotropic curvature (PIC). We say the that a Riemannian manifold $(M^n,g)$, $n\geq 4$, has positive isotropic curvature if $R_{1313}+R_{1414}+R_{2323}+R_{2424}-2R_{1234}>0$ for all orthonormal four-frames $\{e_1,e_2,e_3,e_4\}$. PIC was first studied by Micallef and Moore \cite{MicallefMooreMinimal}.
 
\begin{remark}
We observe that for a Riemannian manifold $(M^n, g)$ with $n \geq 4$, the condition of PIC is \emph{incomparable} with the condition of positive Ricci curvature; neither implies the other, in general.

This can be seen by noting that for any $n \geq 4$, the Riemannian manifold $(\bS^1 \times \bS^n, g_{\bS^1} \oplus g_{\bS^n})$, where $g_{\bS^m}$ denotes the standard metric of radius one on $\bS^m$, has PIC. However, by the Bonnet--Myers theorem, $\bS^1\times\bS^n$ does not support a Riemannian metric with positive Ricci curvature since its fundamental group is infinite. Moreover, for $p,q\geq 2$, $(\bS^p\times\bS^q, g_{\bS^p}\oplus g_{\bS^{q}})$ has positive Ricci curvature (and is Einstein when $p=q$). However, $\bS^p\times\bS^q$ does not support a Riemannian metric with PIC by Micallef--Moore \cite[Main Theorem]{MicallefMooreMinimal}. (See Brendle \cite[Section 7.7]{BrendleRicciFlowTextBook} for a detailed comparison of positive curvature conditions.)
\end{remark}

We now define two curvature conditions that can distinguish the disk among compact, contractible manifolds with boundary.

\begin{enumerate}[label=(C\arabic*)]
    \item A Riemannian metric $g$ on a $(n+1)$-manifold with boundary satisfies curvature condition \ref{C1} if $g$ has PIC and the boundary is 2-convex. \label{C1}
    \item A Riemannian metric $g$ on a $(n+1)$-manifold with boundary satisfies curvature condition \ref{C2} if $g$ satisfies $ng\leq \Ric\leq \frac{1}{2}n(n+1)g$ and the boundary is convex. \label{C2}
\end{enumerate}

\begin{mainthm}\label{T: CurvManwithBoundary}
    Let $X^{n+1}$ be a compact, contractible $(n+1)$-manifold with boundary such that one of the following two conditions holds.
    \begin{enumerate}[label=\emph{(\roman*)}]
        \item $n=4$ or $n\geq 12$ and $X$ supports a Riemannian metric $g$ satisfying \emph{\ref{C1}}. \label{B(i)}
        \item $n\in\{3,4\}$ and $X$ supports a Riemannian metric $g$ satisfying \emph{\ref{C2}}. Furthermore, if $n=4$, assume $\pi_3(X,\partial X)=0$.\label{B(ii)}
    \end{enumerate}
    Then $X$ is homeomorphic to the $(n+1)$-disk.
\end{mainthm}

\begin{remark}\label{r: Ric4}
  We can make the following refinement to the conclusion of \Cref{T: CurvManwithBoundary}. The homeomorphism can be promoted to a diffeomorphism in any of the following cases: when $n=3$ and $X$ is a Mazur manifold, when $n=4$ and $X$ supports a Riemannian metric $g$ satisfying \emph{\ref{C1}}, and when $n\geq 12$. 
   \end{remark}

\begin{remark}
From \Cref{Scalar5}, \Cref{T: CurvManwithBoundary}, and \Cref{Example1}, one can conclude that for a compact, contractible $5$-manifold with boundary $X$ the hypothesis that its interior supports a complete Riemannian metric with uniformly positive scalar curvature is \emph{much more restrictive} than for X, itself, to support a Riemannian metric with positive scalar curvature metric with mean convex (or convex) boundary.
\end{remark}

Now, an interesting result of Wang \cite{HHWangBoundaryConvexity} gives a sufficient condition for compact manifolds with boundary to be contractible. The sufficient curvature condition defined by Wang is the following.
\begin{enumerate}[label=(C3)]
    \item A Riemannian metric $g$ on a  $(n+1)$-manifold with boundary satisfies curvature condition \ref{C3} if $g$ satisfies 
    \[
\hspace{.5in} \Ric_X >0,\,\, \lambda\cdot\left(\frac{\vol(\partial X)}{\omega_{n}}\right)^{n}> 1-\delta(n),\,\, \text{ and } \,\,\frac{\Ric_{\partial X}}{n-1}\cdot \left(\frac{\vol(\partial X)}{\omega_{n}}\right)^{n}>1-\delta(n),
\] 
where $\lambda$ is the smallest eigenvalue of the second fundamental form of $\partial X$, $\omega_{n}$ is the volume of the standard unit round $n$-sphere, $\delta(n)\in(0,1)$ is a constant depending on $n$, and $\Ric_{\partial X}$ is the induced intrinsic Ricci curvature of $\partial X$.\label{C3}
\end{enumerate}

We note that there exist compact manifolds with boundary that are not contractible and support Riemannian metrics with positive Ricci curvature and convex boundary (see Perelman \cite{PerelmanConstruction} and Sha--Yang \cite{ShaYangExamples, ShaYangPositive}). In particular, Perelman produced a Riemannian metric with positive Ricci curvature and convex boundary on $\bC\bP^2\setminus B^4$, where $B^4$ is a 4-disk. This was later generalized by Burdick \cite[Theorem B]{BurdickMetrics} \cite[Theorem C]{BurdickRicci}.

We now make the following observation which relates the present work to the result of Wang.

\begin{maincor}\label{CurvContr}
    Let $X^{n+1}$, $n\in\{3,4,5\}$, be a compact $(n+1)$-manifold with boundary. Assume $X$ supports a Riemannian metric satisfying $\emph{\ref{C3}}$. Furthermore,
\begin{enumerate}[label=\emph{(\alph*)}]
    \item If $n=4$, assume $\pi_3(X,\partial X)=0$. \label{assumption1}
    \item If $n=5$, assume $\pi_3(X,\partial X)=\pi_4(X,\partial X)=0$. \label{assumption2}
\end{enumerate}
Then $X$ homeomorphic to the $(n+1)$-disk.
\end{maincor}

\begin{remark}\label{r: CurvContr}
   We can make the following refinement to the conclusion of  \Cref{CurvContr}. The homeomorphism can be promoted to a diffeomorphism in any of the following cases: when $n=3$ and $X$ is a Mazur manifold, and when  $n=5$.
\end{remark}

\subsection{Closed manifolds supporting metrics with positive curvature conditions} \label{closedmanifoldsubsec} To frame these results properly, we should recall what is already understood in the case of closed manifolds, where the landscape is considerably more developed.
\subsubsection{Scalar Curvature}\label{subsubsecSC}
In dimension two, the Gauss--Bonnet theorem implies that a closed surface supporting a Riemannian metric with positive (scalar) curvature must be diffeomorphic to either the 2-sphere or the real projective plane. 

In dimension three, the foundational work by Schoen--Yau \cite{SchoenYauOntheStructure}, Gromov--Lawson \cite{GromovLawsontheclassification}, and Perelman \cite{PerelmanRicciFlowwithSurgery} establishes that a closed 3-manifold supports a Riemannian metric with positive scalar curvature if and only if it is diffeomorphic to a connected sum of the form
\[
\bS^3\#\left(\#_{j=1}^J \bS^3/\Gamma_j\right) \# \left(\#_{k=1}^K \bS^2 \times \bS^1\right),
\]
where each $\Gamma_j$ is a nontrivial finite subgroup of $\mathrm{SO}(4)$ acting freely on $\bS^3$.

In higher dimensions, there remains topological restrictions imposed by positive scalar curvature. Chodosh, Li, and Liokumovich \cite{ChodoshLiLiokumovich45} showed that if a closed 4-manifold $M$ satisfies $\pi_2(M) = 0$ and supports a Riemannian metric with positive scalar curvature, then there is a finite cover $\hat{M}$ of $M$ that is homotopy equivalent to either $\bS^4$ or a connected sum of finitely many copies of $\bS^3 \times \bS^1$. Similarly, in dimension five, Chodosh, Li, and Liokumovich \cite{ChodoshLiLiokumovich45} proved that if a closed manifold $M$ satisfies $\pi_2(M) =0$, $\pi_3(M) = 0$, and supports a Riemannian metric with positive scalar curvature, then there is a finite cover of $M$ that is homotopy equivalent to either $\bS^5$ or a connected sum of finitely many copies of $\bS^4 \times \bS^1$. (See also the work of Chen--Chu--Zhu    \cite{ChenChuZhuPositivescalarcurvaturemetrics}.)

\subsubsection{Isotropic Curvature}
Manifolds supporting Riemannian metrics with PIC were first studied by Micallef and Moore \cite{MicallefMooreMinimal}. They showed a simply connected, closed $n$-manifold, $n\geq 4$, which supports a Riemannian metric with PIC is homeomorphic to the $n$-sphere. 

Furthermore, a complete classification of manifolds supporting Riemannian metrics with PIC is known in dimension four by the work of Chen, Tang, and Zhu \cite{ChenTangZhuCompleteClassification} (see also Hamilton \cite{HamiltonFourManifoldsPIC}). In particular, they showed that a connected, closed 4-manifold supports a Riemannian metric with PIC if and only if it is diffeomorphic to a connected sum of the form
\[
\bS^4 \# \left(\#_{j=1}^J \bR\bP^4\right) \# \left(\#_{k=1}^K (\bS^3 \times \bR)/G_k\right),
\]
where each $G_k$ is a cocompact discrete
subgroup of the isometry group of the standard metric on $\bS^3\times\bR$ acting freely on $\bS^3\times\bR$.

Finally, we note that very recent work of Huang \cite{HuangClassification,HuangCompact}, building on Brendle \cite{BrendleRicciFlowPIC} and Chen--Tang--Zhu \cite{ChenTangZhuCompleteClassification}, states a connected, closed $n$-manifold, $n\geq 12$, supports a Riemannian metric with PIC if and only if it is diffeomorphic to a connected sum of the form 
\[
\bS^n\#\left(\#_{j=1}^J  \bS^n/\Gamma_j\right) \# \left(\#_{k=1}^K (\bS^{n-1} \times \bR)/G_k\right),
\]
where each $G_k$ is a cocompact discrete
subgroup of the isometry group of the standard metric on $\bS^{n-1}\times\bR$ acting freely on $\bS^3\times\bR$ and each $\Gamma_j$ is a nontrivial finite subgroup of the isometry group of the standard metric on $\bS^{n}$ acting freely on $\bS^n$. 

\bigskip

The paper is organized as follows. In \Cref{prelims}, we collect important propositions from algebraic topology and Riemannian geometry, especially results about $\mu$-bubbles.  In \Cref{proofs}, we prove \Cref{Scalar5}, \Cref{T: CurvManwithBoundary}, and \Cref{CurvContr}. The main approach underlying the proofs is to apply methods from geometric analysis (e.g. $\mu$-bubbles) to impose restrictions on the boundary of the manifold, and then to use algebraic topology to complete the arguments.

\section{Acknowledgments}
This project has received funding from the European Research Council (ERC) under the European Union’s
Horizon 2020 research and innovation programme (grant agreement No. 947923). The author thanks Alessandro Carlotto for his interest in this work and his comments on an early draft, which greatly helped improve the style of the paper.

\section{Preliminaries}\label{prelims}

In this section, we collect relevant propositions from Riemannian geometry and topology that will be used to prove our main theorems.
\subsection{Algebraic Topology} The following statements address topological constraints on the boundaries of contractible manifolds and integral homology spheres. 

 \begin{prop}\label{poincarehomology}
    Let $X^{n+1}$ be a compact, contractible $(n+1)$-manifold with boundary, then $\partial X$ is an integral homology sphere, namely, $H_*(\partial X;\bZ)=H_*(\bS^n;\bZ)$.
\end{prop}
\begin{proof}
    As $X^{n+1}$ is contractible, $\pi_i(X)=0$ and $H_i(X;\bZ)=H^i(X;\bZ)=0$ for all $i\geq 1$ and $H_0(X;\bZ)=H^0(X;\bZ)=\bZ$. Recall that the long exact sequence for the integral homology of $(X,\partial X)$ states
    \[
\cdots \to H_{i+1}(\partial X;\bZ) \to H_{i+1}(X;\bZ) \to H_{i+1}(X, \partial X;\bZ) \to H_{i}(\partial X;\bZ) \to H_{i}(X;\bZ) \to \cdots
    \]
    on the other hand Poincaré-Lefschetz duality says $H_{i+1}(X, \partial X;\bZ) \cong H^{n- i}(X;\bZ).$ For $1\leq i\leq n$, the long exact sequence becomes
    \[
    0\to H_{i+1}(X, \partial X;\bZ) \to H_{i}(\partial X;\bZ) \to0.
    \]
    Therefore, $H_{i}(\partial X;\bZ)\cong H_{i+1}(X, \partial X;\bZ)\cong H^{n - i}(X;\bZ)\cong 0$ for $1\leq i\leq n-1$ and  $H_{n}(\partial X;\bZ)\cong H_{n+1}(X, \partial X;\bZ)\cong H^{0}(X;\bZ)\cong \bZ$. 

    Since $X$ is connected, we know that $H_0(X,\partial X; \bZ)=0$; therefore,
    \[
    0 \to H_{0}(\partial X;\bZ) \to H_{0}(X;\bZ)\to 0.
    \]
    We see then that $H_{0}(\partial X;\bZ) \cong H_{0}(X;\bZ)\cong \bZ$.
\end{proof}

The next proposition discusses what manifolds can be a finite cover of an integral homology sphere.

\begin{lemma}\label{euler}
    Let $M^{2n}$ be an integral homology $2n$-sphere. Then a finite cover of $M$ cannot be homotopy equivalent to a connected sum of finitely many copies of $\bS^{2n-1}\times\bS^1$.
\end{lemma}
\begin{proof}
    We proceed by contradiction. Suppose that $N^{2n}$ is a manifold that is homotopy equivalent to a connected sum of finitely many copies of $\bS^{2n-1}\times\bS^1$ and is a $d$-cover of $M$ for some $d\geq 1$. Now consider the Euler characteristic:
    \[
    \chi\left(\#_{k=1}^K(\bS^{2n-1}\times\bS^1)\right)=\chi(N) = d\cdot\chi(M).
    \]
    The first equality holds since the Euler characteristic is a homotopy invariant and the second holds because $N$ is a $d$-cover of $M$.
    
    As $M$ is an integral $2n$-homology sphere, we have $\chi(M) =2$. Also, we note
    \begin{align*}
        \chi\left(\#_{k=1}^K(\bS^{2n-1}\times\bS^1)\right) &= \chi\left(\bS^{2n-1}\times\bS^1\right) + \chi\left(\#_{k=1}^{K-1}(\bS^{2n-1}\times\bS^1)\right) - \chi\left(\bS^{2n}\right)\\
        &= \chi\left(\bS^{2n-1}\right)\cdot\chi\left(\bS^1\right) + \chi\left(\#_{k=1}^{K-1}(\bS^{2n-1}\times\bS^1)\right) - 2 \\
        &=\chi\left(\#_{k=1}^{K-1}(\bS^{2n-1}\times\bS^1)\right) - 2,
    \end{align*}
    where in the first equality we apply the property of the Euler characteristic for connected sums, in the second equality we apply the property of the Euler characteristic for products, and the last equality follows from fact $\chi(\bS^1)=0$.

    Then, by induction $\chi\left(\#_{k=1}^K(\bS^{2n-1}\times\bS^1)\right)=\chi\left((\bS^{2n-1}\times\bS^1)\right)-2(K-1)=-2(K-1)$.
    Therefore, $-K+1=d$, which is a contradiction as $d\geq1$ and $K\geq 1$.
\end{proof}

Next, we record a theorem of Sjerve \cite{SjerveHomologySpheres} on integral homology spheres that are covered by the sphere.
\begin{theorem}\label{coversofhomologyspheres}
    If $M^n$, $n\geq 3$, is an integral homology $n$-sphere which is covered by $\bS^n$, then either $\pi_1(M)=0$ (and, thus, the covering map is the identity) or $n=3$ and $\pi_1(M)$ is the binary icosahedral group.
\end{theorem}
This means the Poincar\'{e} homology 3-sphere (which is $\bS^3/\Gamma$ where $\Gamma$ is the binary icosahedral group) is the only non-trivial integral homology $n$-sphere that is covered by the sphere.
\begin{lemma}\label{boundaryconnectedness}
    Let $X^{n+1}$, $n\in\{4,5\}$, be a compact, contractible $(n+1)$-manifold with boundary. If $n=4$ and $\pi_3(X,\partial X)=0$, then $\pi_2(\partial X)=0.$ If $n=5$, $\pi_3(X,\partial X)=0$ and $\pi_4(X,\partial X)=0$, then $\pi_2(\partial X)=0$ and $\pi_3(\partial X )=0.$ 
\end{lemma}
\begin{proof}
    As $X$ is contractible, we have for $i\geq 1$ that $\pi_i(X)=0$. Consider the long exact sequence of homotopy pairs $(X,\partial X)$.
    \[
        \cdots \to \pi_{i+1}(\partial X) \to \pi_{i+1}(X) \to \pi_{i+1}(X, \partial X) \to \pi_{i}(\partial X) \to \pi_{i}(X) \to \cdots.
    \]
  In particular, for $i\geq 1 $, $0 \to \pi_{i+1}(X,\partial X) \to \pi_{i}(\partial X) \to 0$. Therefore, $\pi_{i+1}(X,\partial X) \cong  \pi_{i}(\partial X)$.
\end{proof}

Lastly, we need the following fact about virtually infinite cyclic groups which will use later. First, we recall the definition of a virtually infinite cyclic group.

\begin{deff}
    A group $G$ is called \emph{virtually infinite cyclic} if $G$ contains $\bZ$
 as subgroup of finite index.
\end{deff}

\begin{prop}\label{PropVirtuallycyclic}
    The abelianization of a virtually infinite cyclic group is non-trivial.
\end{prop}
\begin{proof}
    Let $G$ be a virtually infinite cyclic group. Then there exists a short exact sequence $0\to N\to G\to G/N\to 0$ with a finite normal subgroup $N$ with $G/N$ isomorphic to $\bZ$ or $\bZ_2*\bZ_2$ (see \cite[Lemma 11.4]{Hempel3manifold}). Moreover, we note that the abelianization of $\bZ_2*\bZ_2$ is $\bZ_2\times\bZ_2$. Therefore, there is a non-trivial homomorphism $\varphi:G\to A$ where $A$ is an Abelian group (either $\bZ$ or $\bZ_2\times\bZ_2$). By the universal property of $G^{\mathrm{ab}}$ (the abelianization of $G$) there is a unique homomorphism $\tilde{\varphi}$ such that the following diagram commutes.

\[
\begin{tikzcd}
G \arrow[r, "\varphi"] \arrow[d, "\rho"'] & A \\
G^{\mathrm{ab}} \arrow[ur, dashed, "\tilde{\varphi}"']
\end{tikzcd}
\]

Since $\varphi$ is non-trivial, take $g\in G$ such that $\varphi(g)\neq 0$. Therefore, $\tilde{\varphi}(\rho(g))\neq 0$ which implies $\rho(g)\neq 0$. We conclude that $G^{\mathrm{ab}}$ is non-trivial. 
\end{proof}

\subsection{Riemannian Geometry}\label{RiemSubsection} Now, we record important information about manifolds that support Riemannian metrics with positive scalar curvature. For clarity, we will first define the second fundamental form of the boundary of a Riemannian manifold with boundary. Given $(X^{n+1},g)$ a Riemannian manifold with boundary and $\nu$ the outward-pointing unit normal to $\partial X$. We define the second fundamental form  $A_{\partial X}(U,V)=\langle\nabla_U \nu, V\rangle$ for any $U,V\in T\partial X$. We further note that by convex boundary, we mean all the eigenvalues of   $A_{\partial X}$ are strictly positive. We say the boundary is two-convex if $A(U,U)+A(V,V)>0$ for all orthonormal $U,V\in T\partial X$. We say that the boundary is mean convex if the mean curvature (i.e., the trace of $A_{\partial X}$) is strictly positive. Now, we record a short proposition about pinched Ricci curvature and convex boundary.

\begin{prop}\label{pinchedRicci}
    Let $(X^{n+1},g)$, be a Riemannian manifold with non-empty boundary. Assume that $ng\leq\mathrm{Ric}\leq \frac{1}{2}n(n+1)g$ and the boundary is convex. Then the induced metric $g_{\partial X}$ on $\partial X$ has positive scalar curvature.
\end{prop}
\begin{proof}
    Let $\lambda_i$ be the $n$ eigenvalues of $A_{\partial X}$. By the Gauss equations and the Schoen--Yau rearrangement trick \cite{SchoenYauExistenceofincompressible}, we see that
   \[
        \Scal_{\partial X} =\Scal_X -2\Ric_X(\nu,\nu)-||A_{\partial X}||^2+H_{\partial X}^2\geq n(n+1)-n(n+1)+ \sum_{i\neq j} \lambda_i\lambda_j>0.
    \]
\end{proof}
\begin{remark}
    We note that, by the same argument, one can show that following curvature conditions also imply the induced metric on the boundary has positive scalar curvature.
    \begin{enumerate}[label=\emph{(\alph*)}]
        \item  $ng<\mathrm{Ric}< \frac{1}{2}n(n+1)g$ and the boundary is weakly convex.
         \item  $ng<\mathrm{Ric}\leq \frac{1}{2}n(n+1)g$ and the boundary is weakly convex.
        \item  $ng\leq \mathrm{Ric}< \frac{1}{2}n(n+1)g$ and the boundary is weakly convex.
    \end{enumerate}
         Here, weakly convex means that the eigenvalues of $A_{\partial X}$ are non-negative.
\end{remark}

\subsubsection{$\mu$-Bubbles} 
This section provides a brief overview of positive scalar curvature and $\mu$-bubbles; for a more comprehensive treatment, we refer the reader to Chodosh–Li \cite{ChodoshLiGeneralizedsoap}, Zhu \cite{ZhuWidth}, and Gromov \cite{GromovFourLectures}.

Let $(X^{n+1},g)$, $2\leq n\leq 6$, be a Riemannian $(n+1)$-manifold with boundary. Assume that $\partial X \neq \emptyset$ and that $\partial X$ has at least two boundary components. Now let $\partial X = \partial_-X \sqcup \partial_+X$ where both $\partial_+ X$ and $\partial_-X$ are nonempty and both $\partial_+ X$ and $\partial_-X$ are unions of connected components of the boundary. Now, fix a function $h\in C^\infty(\mathrm{int}(X))$ such that $h\to + \infty$ on $\partial_{+} X$ and $h\to -\infty$ on $\partial_{-} X$. Now, choose an open Caccioppoli set $\Omega_0$ with smooth boundary $\partial \Omega_0 \subset \mathrm{int}(X)$ and $\partial_+X\subset \Omega_0$. Consider the following functional:
\begin{equation}\label{e: mububble}
    \mathcal{A}(\Omega) = \int_{\partial^* \Omega} d\mathcal{H}^{n} - \int_X \left(\chi_{_\Omega} - \chi_{_{\Omega_0}} \right)h\, d\mathcal{H}^{n+1}
\end{equation}
for all $\Omega \in \mathcal{C}:= \{\text{Caccioppoli sets } \Omega\subset M \text{ with } \Omega \triangle \Omega_0 \Subset \mathrm{int}(X)\}$. See \Cref{2d}.
\begin{figure}[H]
\begin{center}
\begin{tikzpicture}

    \draw[thick, black] (-0.5,0) circle (2.5);
    \draw[thick, black] (0,0) circle (1);

    \draw[thick, blue, smooth, samples=100, domain=0:360]
        plot ({1.4*cos(\x) + 0.3*sin(3*\x)}, {1.4*sin(\x) + 0.3*cos(2*\x)});

    \node at (0.8,-1.8) {\tiny \(\textcolor{blue}{\Omega_0}\)};

   \fill[blue!30, opacity=0.6] 
        plot[smooth cycle, samples=100, domain=0:360]
        ({1.4*cos(\x) + 0.3*sin(3*\x)}, {1.4*sin(\x) + 0.3*cos(2*\x)});

    \draw[thick, red, smooth, samples=100, domain=0:360]
        plot ({1.6*cos(\x) + 0.25*sin(4*\x)}, {1.6*sin(\x) - 0.25*cos(3*\x)});
   \fill[red!30, opacity=0.6] 
        plot[smooth cycle, samples=100, domain=0:360]
        ({1.6*cos(\x) + 0.25*sin(4*\x)}, {1.6*sin(\x) - 0.25*cos(3*\x)});
    \node at (-1.9,0.5) {\tiny \(\textcolor{red}{\Omega}\)};

    \fill[white]  plot[smooth cycle, samples=100, domain=0:360] (0,0) circle (1);
    \node at (0,0.75) {\tiny \(\partial_+ X\)};
    \node at (-3.5,0) {\tiny \(\partial_- X\)};
\draw[thick, black] (0,0) circle (1);
\end{tikzpicture}
\end{center}
\caption{Schematic picture of a $\mu$-bubble} \label{2d}
\end{figure}
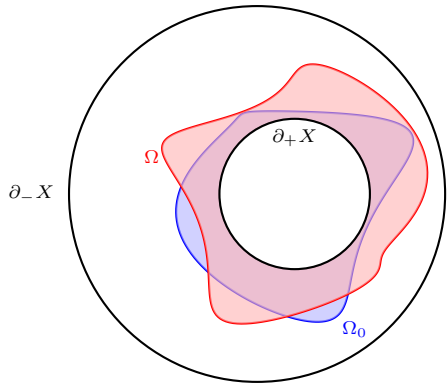

The existence and regularity of a minimizer of $\cA$ among all Caccioppoli
sets is claimed by Gromov \cite[Section 5.2]{GromovFourLectures} and rigorously carried out by Zhu \cite[Proposition 2.1]{ZhuWidth} (cf. \cite[Proposition 12]{ChodoshLiGeneralizedsoap}).

\begin{prop}\label{p: existence}
    For $2\leq n\leq 6$, there exists a smooth $(n+1)$-manifold $\Omega^{n+1}$ which minimizes $\cA$ on $\mathcal{C}$.
\end{prop}

\begin{deff}
    We call a minimizer $\Omega \in \mathcal{C}$ of $\mathcal{A}$ a $\mu$-bubble.
\end{deff}

Next, we record the first and second variations of $\cA$.
 
 \begin{lemma}\label{lemm:1st-var}
Let $\{\Omega_t\}_{t\in(-1,1)}$ be a smooth $1$-parameter family of regions in $\mathcal{C}$ with $\Omega_0 = \Omega$, $\partial \Omega_t=\Sigma_t$, and normal speed $\psi\in C^\infty_c(\Sigma)$ at $t=0$, then 
\[\frac{d}{dt}\cA (\Omega_t)\big|_{t=0}=\int_{\Sigma} (H - h)\psi  \, d\cH^{n},\]
where $\Sigma=\partial \Omega$ and $H$ is the mean curvature of $\Sigma$ which is computed with respect to $\nu$, the outward pointing unit normal to $\partial \Omega=\Sigma$. In particular, a $\mu$-bubble $\Omega$ satisfies a prescribed mean curvature condition
\[
H = h
\]
along $\partial \Omega$. 
\end{lemma}

\begin{lemma}\label{lemm:2nd-var}
Consider a $\mu$-bubble $\Omega$ with $\partial \Omega = \Sigma$. Assume that $\{\Omega_t\}_{t\in(-1,1)}$ is a smooth $1$-parameter family of regions in $\mathcal{C}$ with $\Omega_0 = \Omega$ and normal speed $\psi\in C^\infty_c(\Sigma)$ at $t=0$, then $\frac{d^2}{dt^2}\big|_{t=0}(\cA(\Omega_t))\ge 0$, where 
\[
\frac{d^2}{dt^2}\big|_{t=0}(\cA(\Omega_t))= \int_\Sigma \left(|\nabla_\Sigma \psi|^2-  \tfrac{1}{2} (R_X - R_\Sigma + |A_\Sigma|^2 + h^2 +2\langle \nabla h, \nu \rangle )\psi^2 \right)\, d\cH^{n}.
\]
\end{lemma}    

The next proposition shows that an open manifold with uniformly positive scalar curvature admits a special exhaustion, see \cite[Section 3.7.2]{GromovFourLectures} (cf. \cite[Proposition 3.1]{ChodoshMaximoMukherjeeComplete4}).

\begin{prop}\label{exhaustion}
    Fix a constant $\kappa>0$. Let $(X^{n+1},g)$, $2\leq n\leq 6$, be an open Riemannian $(n+1)$-manifold with scalar curvature $\Scal_X\geq \kappa>0$. There exists an exhaustion $\Omega_1\subset \Omega_2\subset \Omega_3 \subset \cdots $ with $M=\cup_{j=1}^\infty \Omega_i$ where each $\Omega_i$ is a compact codimension zero submanifold with smooth boundary $\partial \Omega_i$ such that $\partial \Omega_i$ supports a Riemannian metric with positive scalar curvature.
\end{prop}

The preceding proposition follows immediately from the iterated use of the Separation Theorem of Gromov \cite[Section 3.7]{GromovFourLectures}, which we state next (see also \cite[Proposition 3.10]{ChodoshMaximoMukherjeeComplete4})
   
\begin{prop}\label{separationthm}
    Fix a constant $\kappa>0$. Let $(X^{n+1},g)$, $2\leq n\leq 6$, be a Riemannian $(n+1)$-manifold with boundary. Assume that $\partial X \neq \emptyset$ and that $\partial X$ has at least two boundary components. Let $\partial X = \partial_-X \sqcup \partial_+X$ where $\partial_\pm X \neq \emptyset$ are unions of connected components of the boundary. Assume that the scalar curvature satisfies $\Scal_X\geq\kappa >0$. Then there is a constant $C(\kappa)=\max\left\{3\pi,\frac{5\pi}{2\kappa}\right\}$ such that if $d(\partial_-X,\partial X_+)> C$, there exists a smooth embedded closed 2-sided hypersurface $\Sigma\subset \mathrm{int}(X) $ that separates $\partial_-X$ from $\partial_+X$ and supports a Riemannian metric with positive scalar curvature.
\end{prop}
For the sake of completeness, we provide a quick sketch of the proof of \Cref{separationthm}.
\begin{proof}
 Fix $0<\eps<\frac{1}{4}$, then there exist smooth functions $d_{\pm}:X\to \bR$ which agree with $d(\cdot,\partial_\pm X):X\to\bR$ in a small neighborhood of $\partial_\pm X$ (respectively) and satisfy $|d_\pm(x)-d(x,\partial_\pm X)|<\eps$ and $|\Lip(d_\pm)|\leq 1+\eps$ for all $x\in X$.   
    
Let $\delta:=d(\partial_+X,\partial_-X)-C>0$. The goal now is to construct a function using $d_\pm$ that goes to $+\infty$ on $\partial_+X$ and $-\infty$ on $\partial_-X$. Define two smooth functions $\tau_-,\tau_+:\bR\to\bR$ with the following properties (see \Cref{tau+_tau_-}).
    \begin{equation}
       \left\{\begin{array}{l}
            \tau_{-}(0)=0 \text{ and } \tau_{-}(t)>0 \text{ for } t>0 \\
             \tau_{-} \text{ is non-decreasing} \\
             \tau_{-}(t)=C \text{ for } t \geq C+\delta\\
             |\Lip(\tau_{-})|<1.
        \end{array}\right.
        \qquad
       \left\{\begin{array}{l}
            \tau_{+}(0)=C \text{ and } \tau_{2}(t)<C \text{ for } t>0 \\
             \tau_{+} \text{ is non-increasing} \\
             \tau_{+}(t)=0 \text{ for } t\geq C+\delta\\
             |\Lip(\tau_{+})|<1.
        \end{array}\right.
    \end{equation}
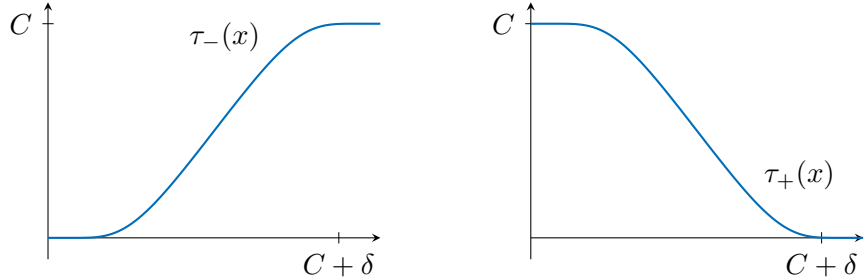
\begin{figure}[H]
\centering
\begin{tikzpicture}
  \begin{groupplot}[
    group style={
      group size=2 by 1,
      horizontal sep=2cm,
    },
    width=6cm, height=5cm,
    domain=0.5:2.5, samples=300,
    xmin=0.5, xmax=2.5, ymin=-0.1, ymax=1.1,
    xtick={0.5,2.25}, xticklabels={$0$,$C+\delta$},
    ytick={0,1}, yticklabels={$0$,$C$},
    axis lines=middle,     
    clip=false,  
    tick style={black},
    every axis x label/.style={at={(axis description cs:0.95,0.03)},anchor=north},
    every axis y label/.style={at={(axis description cs:0.03,0.95)},anchor=east},
  ]

    \nextgroupplot[
    ]
    \addplot[NavyBlue, thick] 
      {ifthenelse(
          x<=0.5, 
          0, 
          ifthenelse(
            x>=2.5, 
            1, 
            exp(-1/(x/2 - 0.25))
            /
            ( exp(-1/(x/2 - 0.25)) 
              + exp(-1/(1 - (x/2 - 0.25))) )
          )
      )} node[pos=0.7, above left, black] {$\tau_-(x)$};

    \nextgroupplot[
    ]
    \addplot[NavyBlue, thick] 
      {ifthenelse(
          x<=0.5,  
          1, 
          ifthenelse(
            x>=2.5, 
            0, 
            1 - (
              exp(-1/(x/2 - 0.25))
              /
              ( exp(-1/(x/2 - 0.25)) 
                + exp(-1/(1 - (x/2 - 0.25))) )
            )
          )
      )}
node[pos=0.7, above right, black] {$\tau_+(x)$};
  \end{groupplot}
\end{tikzpicture}
\caption{Graphs of $\tau_-$ and $\tau_+$}\label{tau+_tau_-}
\end{figure}
\noindent Define now
\begin{equation}
    \rho(x)=(1-\delta)\tau_{-}(d_-(x)) +\delta\tau_{+}(d_+(x)).
\end{equation}
Therefore, $\rho$ has the following properties (see \Cref{rhograph}): 
\begin{enumerate}[label=(\alph*)]
    \item $\rho|_{\partial_-X}=0$.
    \item $\rho|_{\partial_+X}=C$.
    \item $0\leq \rho(x)\leq C$ for $x\in X$.
    \item $|\Lip(\rho)|<1+\eps$.
\end{enumerate}

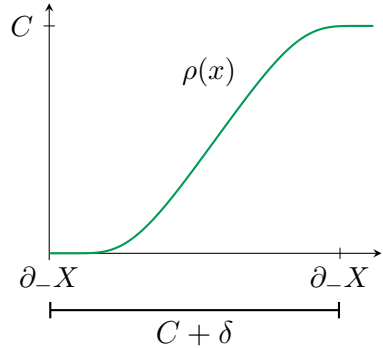
\begin{figure}[H]
    \centering
    \begin{tikzpicture}
  \begin{axis}[
    width=6cm, height=5cm,
    domain=0.5:2.5, samples=300,
    xmin=0.5, xmax=2.5, ymin=-0.03, ymax=1.1,
    xtick={0.5,2.25}, xticklabels=\empty,
    ytick={0,1}, yticklabels={$0$,$C$},
    axis lines=middle,     
    clip=false,  
    tick style={black},
    every axis x label/.style={at={(axis description cs:0.95,0.03)},anchor=north},
    every axis y label/.style={at={(axis description cs:0.03,0.95)},anchor=east},
  ]

     \addplot[ForestGreen, thick, domain=0.5:1.0]
{ifthenelse(
          x<=0.5, 
          0, 
          ifthenelse(
            x>=2.5, 
            1, 
            exp(-1/(x/2 - 0.25))
            /
            ( exp(-1/(x/2 - 0.25)) 
              + exp(-1/(1 - (x/2 - 0.25))) )
          )
      )};
      \addplot[ForestGreen, thick, domain=1.0:2.0] 
{ifthenelse(
          x<=0.5, 
          0, 
          ifthenelse(
            x>=2.5, 
            1, 
            exp(-1/(x/2 - 0.25))
            /
            ( exp(-1/(x/2 - 0.25)) 
              + exp(-1/(1 - (x/2 - 0.25))) )
          )
      )} node[pos=0.7, above left, black] {$\rho(x)$};
       \addplot[ForestGreen, thick, domain=2.0:2.45] 
      {ifthenelse(
          x<=0.5, 
          0, 
          ifthenelse(
            x>=2.5, 
            1, 
            exp(-1/(x/2 - 0.25))
            /
            ( exp(-1/(x/2 - 0.25)) 
              + exp(-1/(1 - (x/2 - 0.25))) )
          )
      )};
      \node[below, black] at (axis cs:0.5, -0.015) {$\partial_- X$};\node[below, black] at (axis cs:2.25, -0.015) {$\partial_- X$};
  \draw[|-|, thick] (axis cs:0.5,-0.25) -- (axis cs:2.25,-0.25);
  \node at (axis cs:1.375,-0.35) {\large $C + \delta$};
  \end{axis}
\end{tikzpicture}
\caption{Graph of $\rho$}\label{rhograph}
\end{figure}
\noindent Finally, define the smooth function $h:X\to \bR$
\begin{equation}
    h(x):=\tan\left(\frac{\rho(x)-\frac{1}{2}C}{C}\pi\right)
\end{equation}
and note that $h\to\pm\infty$ on $\partial_\pm X$, respectively  (see \Cref{hgraph}).
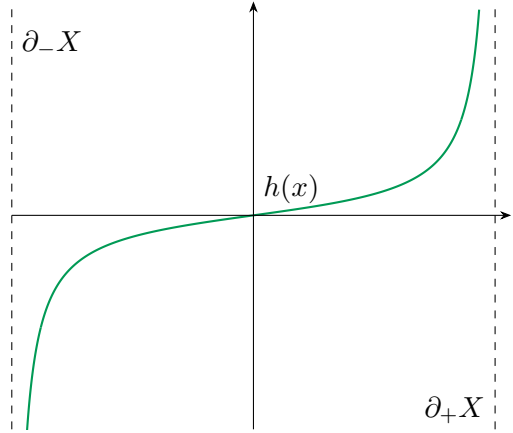
\begin{figure}[H]
\begin{center}
\begin{tikzpicture}
  \begin{axis}[
    axis lines=middle,
    axis x line=none,
    axis y line=none,
    trig format plots=rad,        
    xmin=-0.1, xmax=3.1,
    ymin=-10, ymax=10,
    domain=0.01:2.99,
    samples=1000,
    unbounded coords=jump,
    xtick=\empty, 
    ytick=\empty,
    every axis x label/.style={at={(axis description cs:0.95,0.03)},anchor=north},
    every axis y label/.style={at={(axis description cs:0.03,0.95)},anchor=east},
  ]
    \addplot[ForestGreen, thick, domain=0.01:.18] 
      {tan((x - 1.5)/3 * pi)};
      
    \addplot[ForestGreen, thick, domain=.18:2.79] 
      {tan((x - 1.5)/3 * pi)};
      
    \addplot[ForestGreen, thick, domain=2.79:2.901] 
      {tan((x - 1.5)/3 * pi)};

    \draw[dashed] (axis cs:0,-10) -- (axis cs:0,10);
    \draw[dashed] (axis cs:3,-10) -- (axis cs:3,10);
    \draw[-{Stealth[length=4.25pt, width=3.25pt]}] (axis cs:1.5,-10) -- (axis cs:1.5,10);
     \draw[-{Stealth[length=4.25pt, width=3.25pt]}] (axis cs:0,0) -- (axis cs:3.1,0);

    \node[anchor=west] at (axis cs:0,8) {$\partial_- X$};
    \node[anchor=north east] at (axis cs:3,-8) {$\partial_+ X$};
    \node[anchor=south west] at (axis cs:1.5,0.1) {$h(x)$};
  \end{axis}
\end{tikzpicture}
\end{center}
\caption{A graph of $h$} \label{hgraph}
\end{figure}

With respect to this $h$, we can find a $\mu$-bubble ${\Omega}$ with boundary $\partial\Omega=\Sigma$ by using \Cref{p: existence}.

By \Cref{lemm:2nd-var}, for every $\psi\in C^\infty_c(\Sigma)$, 
\begin{align}
\int_{\Sigma} \left(|\nabla_\Sigma \psi|^2-  \tfrac{1}{2} (R_X - R_\Sigma + |A_\Sigma|^2 + h^2 +2\langle \nabla h, \nu \rangle )\psi^2 \right) d\cH^{n-1}\geq 0.
\end{align}
Therefore,
\begin{align}\label{e: equation}
    \int_{\Sigma} \left(|\nabla_\Sigma \psi|^2+\tfrac{1}{2} R_{\Sigma} \right)\psi^2 d\cH^{n-1}\geq \int_{\Sigma}\frac{1}{2} \left( R_X+|A_\Sigma|^2 + h^2 +2\langle \nabla h, \nu \rangle )\psi^2 \right) d\cH^{n-1}.
\end{align}
Now, we compute:
\begin{align*}
     h^2 +2\langle \nabla h, \nu \rangle &\geq  h^2 -2| \nabla h|\\
     &=\tan^2\left(\frac{\rho(x)-\frac{1}{2}C}{C}\pi\right) -\frac{2|\nabla\rho(x)|\pi}{C}\sec^2\left(\frac{\rho(x)-\frac{1}{2}C}{C}\pi\right)\\
     &\geq \tan^2\left(\frac{\rho(x)-\frac{1}{2}C}{C}\pi\right) -\frac{2 \pi}{C}(1+\eps)\sec^2\left(\frac{\rho(x)-\frac{1}{2}C}{C}\pi\right)\\
     &> \frac{2\pi}{C}(1+\eps)\tan^2\left(\frac{\rho(x)-\frac{1}{2}C}{C}\pi\right) -\frac{2\pi}{C}(1+\eps) \sec^2\left(\frac{\rho(x)-\frac{1}{2}C}{C}\pi\right)\\
     &\geq -\frac{5\pi}{2C},
\end{align*}
where the penultimate inequality follows since $C>2\pi$ and the final inequality follows by the trigonometric identity $\tan^2(\theta)-\sec^2(\theta)=-1$ and $\eps\in\left(0,\frac{1}{4}\right)$.

Therefore, by \eqref{e: equation}, we have
\[
    \int_{\Sigma} \left(|\nabla_\Sigma \psi|^2+\tfrac{1}{2} R_{\Sigma} \right)\psi^2 d\cH^{n-1}>\frac{1}{2} \int_{\Sigma} \left(R_X -\frac{5\pi}{2C} \right)\psi^2 d\cH^{n-1}.
\]
Since $C(\kappa)=\max\left\{3\pi,\frac{5\pi}{2\kappa}\right\}$, we conclude that $\int_{\Sigma} \left(|\nabla_\Sigma \psi|^2+\tfrac{1}{2} R_{\Sigma} \right)\psi^2 d\cH^{n-1}>0$. When $n=2$, we see from the Gauss--Bonnet formula and choosing $\psi\equiv 1$ that $\Sigma$ supports a Riemannian metric with positive scalar curvature. For $n\geq 3$, we can use the first eigenfunction of the conformal Laplacian to produce a Riemannian metric on $\Sigma$ with positive scalar curvature.
\end{proof}

Next, we would like to record a result of Chodosh--Li--Liokumovich \cite{ChodoshLiLiokumovich45}.

\begin{theorem}\label{dim45mapping}\label{ClassifyChodoshLiLIokomovich}
    Suppose $(M^n,g)$ is a closed $n$-manifold supporting a Riemannian metric with positive scalar curvature, and there exists a map with non-zero degree, $f:M\to N$, to an $n$-manifold $N$ satisfying
    \begin{enumerate}
        \item $n=4$ and $\pi_2(N)=0$, or
        \item $n=5$ and $\pi_2(N)=\pi_3(N)=0$.
    \end{enumerate}
    Then there exists a finite cover $\hat{N}$ of $N$ which is homotopy equivalent to $\bS^n$ or a connected sum of finitely many copies of $\bS^{n-1}\times \bS^1$.
\end{theorem}

\begin{remark}
    By choosing $f$ to be the identity map, one recovers the classification result of Chodosh, Li, and Liokumovich for 4- and 5-manifolds, that was discussed in \Cref{subsubsecSC}.
\end{remark}

We now have the following important new corollary to \Cref{dim45mapping}.

\begin{cor}\label{CorChodoshLiLiokomovich}
    Let $X^{n+1}$, $n\in\{4,5\}$,  be a compact, contractible $n$-manifold with boundary such that the interior of $X$ supports a complete Riemannian metric with uniformly positive scalar curvature. If $n=4$, further assume $\pi_3(X,\partial X)=0$, and if $n=5$, further assume $\pi_3(X,\partial X)=0$ and $\pi_4(X,\partial X)=0$. Then the boundary $\partial X$ has a finite cover that is homotopy equivalent to $\bS^n$ or a connected sum of finitely many copies of $\bS^{n-1}\times \bS^1$.
\end{cor}
\begin{proof}
   By \Cref{poincarehomology},  we have that $\partial X$ is an integral homology sphere; therefore, $\partial X$ is connected and orientable. By the collar neighborhood theorem, there is a closed set $U$ that contains the boundary, diffeomorphic to $\partial X \times [-1,1]$. For sufficiently large $i$, the exhaustion from \Cref{exhaustion} will have $\partial \Omega_i$ a hypersurface in $U$ separating $\partial X \times \{-1\}$ from $\partial X \times \{1\}$ and $\partial \Omega_i$ supports a Riemannian metric with positive scalar curvature.

   Consider now the projection map $\pi:U\to \partial X$ and its restriction $\pi|_{\partial \Omega_i}:\partial \Omega_i\to \partial X$. If $\omega$ is the volume form on $\partial X$, then $\int_{\partial\Omega_i}(\pi|_{\partial \Omega_i})^*\omega\neq 0$ and so $\pi|_{\partial \Omega_i}$ has non-zero degree. (One can see this since $\partial\Omega_i\subset U$ separates $\partial X \times \{-1\}$ from $\partial X \times \{1\}$, which implies there is an $(n+1)$-manifold $W\subseteq U$ with $\partial W= \partial \Omega_i \sqcup \partial X$.) 
   
   By \Cref{boundaryconnectedness}, if $n=4$, then $\pi_2(\partial X)=0$, and if $n=5$, then $\pi_2(\partial X)=0$ and $\pi_3(\partial X )=0.$  Now, apply \Cref{dim45mapping} to finish the proof.
\end{proof}
We would like to compare the proceeding corollary to a consequence of the following statement.

\begin{prop}\label{p: separtingpsc}
   Let $n\geq 2$ and  $n\neq 4$ and assume $Y^n$ is an $n$-dimensional closed, connected, oriented manifold and $\Sigma\subset Y\times [-1,1]$ is a closed embedded hypersurface that separates $Y\times \{-1\}$ from $Y\times\{1\}$. If $\Sigma$ supports a Riemannian metric with positive scalar curvature, then $Y$ supports a Riemannian metric with positive scalar curvature.
\end{prop}

For $n\in\{2, 3\}$, the proof of \Cref{p: separtingpsc} follows from the classification of closed 2- and 3-manifolds supporting a Riemannian metric with positive scalar curvature. Therefore, for $n=2$, it follows from the Gauss--Bonnet formula (cf. \cite[Proposition 2.1]{ChodoshMaximoMukherjeeComplete4}). For $n=3$, it follows from Gromov--Lawson \cite{GromovLawsontheclassification}, Schoen--Yau \cite{SchoenYauOntheStructure}, and Perelman \cite{PerelmanRicciFlowwithSurgery} (cf. \cite[Proposition 2.2]{ChodoshMaximoMukherjeeComplete4}). Finally, R\"ade proved the statement for $n\geq 5$ \cite[Proposition 2.17]{RadeScalar}. Moreover, there is a counterexample for $n=4$ \cite[Remark 1.25]{RosenbergManifolds}. Now, we state a corollary to \Cref{p: separtingpsc} which should be compared with \Cref{CorChodoshLiLiokomovich}. 

\begin{cor}\label{RadeCor}
     Let $X^{n+1}$, $n\geq 2$ and $n\neq 4$, be an orientable, compact $n$-manifold with boundary. Assume that the interior of $X$ supports a complete Riemannian metric with uniformly positive scalar curvature. Then each boundary component of $\partial X$ supports a Riemannian metric with positive scalar curvature.
\end{cor}

This can be proved in a similar way as \Cref{CorChodoshLiLiokomovich}, where one uses \Cref{p: separtingpsc} instead of \Cref{dim45mapping} and \Cref{boundaryconnectedness}. \Cref{RadeCor} is valid in all dimensions except four and guarantees that each boundary component supports a Riemannian metric of positive scalar curvature (which is a restriction on the topology of the boundary). Now, \Cref{CorChodoshLiLiokomovich} provides a way to extend this result to dimension four (necessarily under further hypotheses) in order give a topological restriction on the boundary.

\subsubsection{Positive Isotropic Curvature} Here we will collect information about positive isotropic curvature.
Our choice of notation will follow that of Brendle \cite{BrendleRicciFlowTextBook}. 

Let $\nabla$ be the Levi-Civita connection on a Riemannian manifold $(M^n,g)$ and recall that the Riemann curvature tensor is 
\[
R(U,V)Z=\nabla_U\nabla_VZ-\nabla_V\nabla_UZ-\nabla_{[U,V]}Z,
\]
for all vector fields $U, V,$ and $Z$.
For each point $p\in M$, the curvature operator $\mathfrak{R}:\bigwedge^2T_pM\times\bigwedge^2T_pM\to \bR$ is defined by
\[
\mathfrak{R}(U\wedge V,Z\wedge W)=-g(R(U,V)Z,W).
\]
Denote $\mathfrak{R}(e_i\wedge e_j,e_k\wedge e_l)=:R_{ijkl}.$
Now, we can define PIC. 

\begin{deff}[Positive Isotropic Curvature]\label{PIC def}
A manifold $(M^n,g)$, $n\geq4$, has \emph{PIC} if
\[
R_{1313}+R_{1414}+R_{2323}+R_{2424}-2R_{1234}>0
\]
for all orthonormal four-frames $\{e_1,e_2,e_3,e_4\}$. 
\end{deff}

Now we would like to record a deformation result for PIC, which was proven by Chow \cite[Main Theorem 2]{ChowPositivity}.
\begin{theorem}\label{ChowTheoremdeform}
    Let $(X^{n+1},g)$ be a compact Riemannian $(n+1)$-manifold with boundary. If $(X,g)$ has PIC and 2-convex boundary, then there exists a Riemannian metric $\tilde{g}$ on $X$ such that $(X,\tilde{g})$ has PIC and totally geodesic boundary.
\end{theorem}

Finally, we conclude this section with a result about transferring the PIC condition on a Riemannian manifold with boundary to its boundary with the induced metric.

\begin{prop}\label{Propboundarygeodesic}
     Let $(X^{n+1},g)$ be a compact Riemannian manifold with boundary. If $n\geq4$ and $(X,g)$ has PIC and totally geodesic boundary, then the induced metric on $\partial X$ has PIC.
\end{prop}
\begin{proof}
    Let $(X^{n+1},g)$, $n\geq 4$, be a compact Riemannian $(n+1)$-manifold with boundary. Let $g_{\partial X}$ denote the induced metric on $\partial X$. The intrinsic Riemannian curvature tensor of $\partial X$ is denoted by $\mathrm{Riem}_{\partial X}$, and the second fundamental form of $\partial X$ with respect to the outward-pointing unit normal vector $\nu$ is denoted by $A_{\partial X}$.
 The Gauss--Codazzi equations state:
    \begin{align*}
        g(\mathrm{Riem}(U,V)Z,W)&=g_{\partial X}(\mathrm{Riem}_{\partial X}(U,V)Z,W)-A_{\partial X}(U,W)A_{\partial X}(V,Z)\\
        &\qquad+A_{\partial X}(U,Z)A_{\partial X}(V,W).
    \end{align*}
    As $\partial X$ is totally geodesic, then $g(\mathrm{Riem}(U,V)Z,W)=g_{\partial X}(\mathrm{Riem}_{\partial X}(U,V)Z,W)$ and in terms of the curvature operator 
    \begin{align*}
       \mathfrak{R}(U\wedge V, Z\wedge W)&= - g(\mathrm{Riem}(U,V)Z,W)= -  g_{\partial X}(\mathrm{Riem}_{\partial X}(U,V)Z,W) \\&= \mathfrak{R}_{\partial X}(U\wedge V, Z\wedge W), 
    \end{align*}
    where $\mathfrak{R}_{\partial X}$ is the curvature operator on $\partial X$. Since $(X,g)$ has PIC, we conclude by definition that the induced metric on $\partial X$ has PIC.
\end{proof}

\section{Proofs of Main Results}\label{proofs}
In this section, we will present a few auxiliary propositions and the proofs of the main results. We will begin with the proof of \Cref{Scalar5}.

\begin{proof}[Proof of \Cref{Scalar5}]
 Let $M^5$ be the interior of a compact, contractible 5-manifold with boundary $X$ such that $\pi_3(X,\partial X)=0$. Assume $M$ supports a complete Riemannian metric of uniformly positive scalar curvature. By \Cref{CorChodoshLiLiokomovich}, we have that $\partial X$ has a finite cover that is homotopy equivalent to $\bS^4$ or a connected sum of finitely many copies of $\bS^{3}\times \bS^1$. By \Cref{poincarehomology},  we have that $\partial X$ is an integral homology sphere. By \Cref{euler}, we know that a finite cover of an integral homology sphere is never homotopy equivalent to a connected sum of finitely many copies of $\bS^{3}\times \bS^1$. Therefore, $\bS^4$ is a finite cover of $\partial X$.
 
 Now, by \Cref{coversofhomologyspheres}, we know that $\partial X$ is a simply connected integral homology 4-sphere and so, by the Hurewicz Theorem, $\partial X$ is a homotopy 4-sphere. By Freedman \cite{FreedmanTopologyof4}, we see that $\partial X$ is homeomorphic to $\bS^4$. We conclude that $X$ is homeomorphic to the 5-disk by Milnor \cite[\S9 Proposition C]{MilnorLectures}. Therefore, $M$ is homeomorphic to $\bR^5$ and so is diffeomorphic to $\bR^5$ since $\bR^5$ has unique smooth structure up to diffeomorphism by Stallings \cite[Corollary 5.2]{StallingsThepiecewise}.
 \end{proof}

Next, we present four propositions that will be useful for the proofs of \Cref{T: CurvManwithBoundary} and \Cref{CurvContr}.

\begin{prop}\label{PropB1}
     Let $X^{n+1}$, $n=4$ or $n\geq 12$ be a compact, contractible $(n+1)$-manifold with boundary such that $X$ supports a Riemannian metric $g$ satisfying \emph{\ref{C1}}. Then $\partial X$ is diffeomorphic to the $n$-sphere.
\end{prop}
\begin{proof}
 By \Cref{ChowTheoremdeform} and \Cref{Propboundarygeodesic}, we know that $\partial X$ is a closed $n$-manifold that supports a Riemannian metric with PIC. Moreover, by \cite[Main Theorem]{ChenTangZhuCompleteClassification} or \cite[Theorem 1.1]{HuangCompact}, we know that $\partial X$ is diffeomorphic to a connected sum of the form 
 \[
 \bS^n\#(\#_{j=1}^J \bS^n/\Gamma_j)\# (\#_{k=1}^K  (\bS^{n-1}\times \bR)/G_k),
 \]
 where each $G_k$ is a cocompact discrete
subgroup of the isometry group of the standard metric on $\bS^{n-1}\times\bR$ acting freely on $\bS^{n-1}\times\bR$ and each $\Gamma_j$ is a nontrivial finite subgroup of the isometry group of the standard metric on $\bS^{n}$ acting freely on $\bS^n$. Therefore, we have that the fundamental group is

\[
\pi_1(\partial X)= \Gamma_1*\cdots*\Gamma_J * G_1*\cdots *G_K.
\]

\begin{claim}
    The groups $G_k$ are virtually infinite cyclic.
\end{claim}   
\begin{proof}[Proof of Claim]
    Let $R_k=(\bS^3\times \bR)/G_k)$ and note $R_k$ is a closed, connected manifold. By the Milnor--\v Svarc lemma \cite{MilnorAnote,SvarcAvolume}, $\pi_1(R_k)=G_k$ is finitely generated and there exists a distance function $d_k$ on $G_k$ such that the metric space $(G_k, d_k)$ is quasi-isometric to $(\bS^3\times \bR, d_{\bS^3\times\bR})$, where $d_{\bS^3\times\bR}$ is the distance function induced from the standard metric on $\bS^3\times \bR$.  Now, $(\bS^3\times \bR, d_{\bS^3\times\bR})$ is quasi-isometric $(\bR, d_\bR)$ which is quasi-isometric to $(\bZ,d_\bZ)$, where both $d_\bR$ and $d_\bZ$ are the standard distance functions on $\bR$ and $\bZ$, respectively. Therefore, $G_k$ is a group with two ends \cite[Corollary 2.3]{BrickQuasi}. Finally, a group with two ends is virtually infinite cyclic \cite{Freudenthal, Hopf} (see also \cite[Lemma 4.1]{WallPoincare}).
\end{proof}
Now, by \Cref{PropVirtuallycyclic} we know that the abelianization of a virtually infinite cyclic group is non-trivial. However, by \Cref{poincarehomology}, we know that $\partial X$ is an integral homology sphere. Therefore, the abelianization of $\pi_1(\partial X)$ must vanish. Therefore, $K=0$.

Consider $\bS^n/\Gamma_j$ in the connected sum description of $\partial X$. We will now examine two cases when $n$ is even and when $n$ is odd.

\noindent\underline{{\textsc{Case 1:}}} If $n$ is even, then $\bS^n/\Gamma_j$ is homeomorphic to $\bR\bP^n$. Therefore, $\pi_1(\bS^n/\Gamma_j)=\bZ_2$ and $H_1(\bS^n/\Gamma_j)=\bZ_2$.  Since $\partial X$ is an integral homology sphere, we conclude that $J=0$ and $\partial X$ is diffeomorphic to the $n$-sphere.

\noindent\underline{{\textsc{Case 2:}}} Now for $n$ odd, we note that $\bS^n/\Gamma_j$ are compact, connected, and orientable. (The orientability of $\bS^n/\Gamma_j$ can be proven by noting that $\bS^n/\Gamma_j$ supports a Riemannian metric with positive sectional curvature and applying Synge’s theorem.) Therefore, for $2\leq i\leq n-1$,
\[
0=H_i(\partial X)=H_i(\bS^n/\Gamma_1)\oplus\cdots\oplus H_i(\bS^n/\Gamma_J).
\]
Thus, each $\bS^n/\Gamma_j$ is an integral homology sphere which is covered by the sphere. By the assumption $n=4$ or $n\geq 12$, \Cref{coversofhomologyspheres}, and the Hurewicz Theorem, we conclude that $\Gamma_j$ are trivial and $J=0$. Therefore, $\partial X$ is diffeomorphic to the $n$-sphere.

Thus, in either case, $\partial X$ is diffeomorphic to the $n$-sphere.
\end{proof}

\begin{prop}\label{PropB2}
    Let $X^{n+1}$, $n\in\{3,4\}$ be a compact, contractible manifold with boundary such that $X$ supports a Riemannian metric $g$ satisfying \emph{\ref{C2}}. Furthermore, if $n=4$, assume $\pi_3(X,\partial X)=0$. Then $\partial X$ is homeomorphic to the $n$-sphere.
\end{prop}

\begin{proof}
    By \Cref{pinchedRicci}, we know that $\partial X$ supports a Riemannian metric with positive scalar curvature, and by \Cref{poincarehomology}, we have that $\partial X$ is an integral homology sphere.

    If $n=3$, we have by \cite[Corollary 2.3 and Proposition 4.2]{ChodoshMaximoMukherjeeComplete4} that $\partial X$ is diffeomorphic to the 3-sphere. We will quickly outline this proof. By Perelman \cite{PerelmanRicciFlowwithSurgery}, we know that $\partial X$ is a connected sum of the form 
\[
\bS^3\#(\#_{j=1}^J \bS^3/\Gamma_j)\# (\#_{k=1}^K  \bS^2\times \bS^1),
\]
where $\Gamma_j$ is a nontrivial finite subgroup of $\mathrm{SO}(4)$ acting freely on $\bS^3$. Therefore, since $\partial X$ is an integral homology sphere we have 
\[
\partial X \cong\bS^3\#(\#_{j=1}^JP)\#(\#_{k=1}^K -P),
\]
where $P$ represents the Poincar\'e homology sphere and $-P$ is $P$ with the opposite orientation. By applying the Heegard-Floer $d$-invariant \cite[Theorem 1.2, Proposition 4.2, Proposition 4.3, Section 8.1, and Proposition 9.9]{OzsvathSzaboAbsolutely} one concludes $L=M$. Then, by a theorem of Taubes \cite{TaubesGauge}, we conclude  $L=M=0$. Thus, $\partial X$ is $\bS^3$.

If $n=4$, then since $\pi_3(X,\partial X)=0$, we have by \Cref{boundaryconnectedness} that $\pi_2(\partial X)=0$. Thus, as $\partial X$ supports a Riemannian metric with positive scalar curvature, we deduce from \Cref{ClassifyChodoshLiLIokomovich} that a finite cover of $\partial X$ is homotopy equivalent to $\bS^4$ or a connected sum of finitely many copies of $\bS^{3}\times \bS^1$. By \Cref{euler}, this finite cover cannot be homotopy equivalent to a connected sum of finitely many copies of $\bS^{3}\times \bS^1$. Thus, $\partial X$ is covered by $\bS^4$ and is an integral homology sphere; therefore, by \Cref{coversofhomologyspheres}, we conclude that $\pi_1(\partial X)=0$. By the Hurewicz Theorem, $\partial X$ is a homotopy sphere, and by Freedman \cite{FreedmanTopologyof4} $\partial X$ is homeomorphic to $\bS^4$.
\end{proof}

\begin{prop}\label{PropC1}
     Let $X^{n+1}$, $n\in\{3,4,5\}$, be a compact $(n+1)$-manifold with boundary. Assume $X$ supports a Riemannian metric satisfying $\emph{\ref{C3}}$. Additionally:
\begin{enumerate}[label=\emph{(\alph*)}]
    \item If $n=4$, assume $\pi_3(X,\partial X)=0$. 
    \item If $n=5$, assume $\pi_3(X,\partial X)=\pi_4(X,\partial X)=0$.
\end{enumerate}
Then $\partial X$ homeomorphic to the $n$-sphere for $n\in\{3,4,5\}$; furthermore, when $n\in\{3,5\}$, the homeomorphism can be promoted to a diffeomorphism.
\end{prop}
\begin{proof}
    From \cite{HHWangBoundaryConvexity}, we know that $X$ is contractible, and from \ref{C3} we know that the induced metric on $\partial X$ has positive Ricci curvature. 

When $n=3$, by Hamilton \cite{HamiltonThreeManifoldPosRic}, we know that $\partial X$ is a quotient of the 3-sphere.
By \Cref{poincarehomology}, we know that the abelianization of $\pi_1(\partial X)$ is trivial. Therefore, by \Cref{coversofhomologyspheres}, $\partial X$ is homeomorphic to either $\bS^3$ or the Poincar\'e homology sphere. Now, we know that the Heegard-Floer $d$-invariant, $d(\partial X)$, vanishes since $\partial X$ is the boundary of a contractible manifold \cite[Theorem 1.2 and Proposition 9.9]{OzsvathSzaboAbsolutely}. However, the $d$-invariant of the Poincar\'e homology sphere is non-trivial \cite[Section 8.1]{OzsvathSzaboAbsolutely}. Thus, $\partial X$ is diffeomorphic to $\bS^3$.

 When $n\in \{4,5\}$, the topological assumptions, \emph{\ref{assumption1} \ref{assumption2}}, imply via \Cref{boundaryconnectedness} that $\partial X$ satisfies the hypotheses of \Cref{ClassifyChodoshLiLIokomovich}; therefore, a finite cover of $\partial X$ is homotopy equivalent to $\bS^n$ or a connected sum of finitely many copies of $\bS^{n-1}\times\bS^1$. As $\partial X$ supports a Riemannian metric with positive Ricci curvature, we conclude that a finite cover of $\partial X$ cannot be homotopy equivalent to a connected sum of finitely many copies of $\bS^{n-1}\times\bS^1$ (since $\pi_1(\bS^{n-1}\times\bS^1)$ is not finite). Therefore, a finite cover of $\partial X$ is homotopic to $\bS^n$. When $n=4$, we conclude that $\partial X$ is homeomorphic to $\bS^4$ by Freedman \cite{FreedmanTopologyof4} and when $n=5$ that $\partial X$ is diffeomorphic to $\bS^5$ by \cite[\S 9 Proposition B]{MilnorLectures} (see also Smale \cite{SmaleGeneralized, SmaleN}).
\end{proof}

Next, we would like to include a proof that justifies \Cref{r: Ric4} and \Cref{r: CurvContr} --- a Mazur manifold whose boundary is $\bS^3$ is diffeomorphic to the 4-disk. We give the proof from \cite[page 10]{ChodoshMaximoMukherjeeComplete4}.
\begin{prop}\label{PropMazur}
    A Mazur manifold $X$ whose boundary is $\bS^3$ is diffeomorphic to the 4-disk.
\end{prop}
\begin{proof}
   Recall that a Mazur manifold $X$ is a compact, contractible 4-manifold with boundary admitting a (smooth) handle decomposition with one 0-handle, one 1-handle, and one 2-handle. Observe that the 2-handle is attached along a knot on the boundary of the 1-handle which is $\bS^1\times\bS^2$. Thus, $\partial X$ which is homeomorphic to $\bS^3$ is obtained by performing surgery along a knot $K$ in $\bS^1\times\bS^2$. Gabai’s property R theorem \cite{GabaiFoliations} states that $K$ is smoothly isotopic to the $\bS^1$ factor of $\bS^1\times\bS^2$. In the 4-dimensional handle picture, the attaching sphere of the 2-handle intersects the belt sphere of the 1-handle geometrically once, allowing us to cancel the 1- and 2-handle smoothly. We conclude $X$ must be diffeomorphic to the 4-disk.
\end{proof}

Finally, we present the proofs of \Cref{T: CurvManwithBoundary} and \Cref{CurvContr}.

\begin{proof}[Proof of \Cref{T: CurvManwithBoundary}]\label{Proof of A}
Let $X^{n+1}$ is a compact, contractible manifold with boundary.

Assume \emph{\ref{B(i)}} holds, then by \Cref{PropB1}, $\partial X$ is diffeomorphic to the $n$-sphere. By Milnor \cite[\S9 Proposition A and Proposition C]{MilnorLectures} (see also Smale \cite{SmaleGeneralized, SmaleN}), we conclude that $X$ is diffeomorphic to the $(n+1)$-disk.

Assume \emph{\ref{B(ii)}} holds, then by \Cref{PropB2}, $\partial X$ is homeomorphic to the $n$-sphere. When $n=3$, we conclude, by Freedman \cite{FreedmanTopologyof4}, that $X$ is homeomorphic to the $4$-disk. When $n=3$ and $X$ is a Mazur manifold, then $X$ is diffeomorphic to the $4$-disk by \Cref{PropMazur}. When $n=4$, $X$ is homeomorphic to the 5-disk, by \cite[\S 9 Proposition C]{MilnorLectures}. 
\end{proof}

\begin{proof}[Proof of \Cref{CurvContr}]
By \Cref{PropC1}, we know that $\partial X$ is homeomorphic to the $n$-sphere. 
 When $n=3$, we conclude, by Freedman \cite{FreedmanTopologyof4}, that $X$ is homeomorphic to the $4$-disk. When $n=3$ and $X$ is a Mazur manifold, then $X$ is diffeomorphic to the $4$-disk by \Cref{PropMazur}. When $n=4$, $X$ is homeomorphic to the $5$-disk, by \cite[\S 9 Proposition C]{MilnorLectures}. When $n=5$, we conclude that $X$ is diffeomorphic to the $6$-disk, by \cite[\S 9 Proposition A]{MilnorLectures} (see also Smale \cite{SmaleGeneralized, SmaleN}).
\end{proof}

\printbibliography
\end{document}